\documentclass[12pt]{amsart}

\usepackage{pgf}
\usepackage{graphicx}
\usepackage{amsmath, amsthm, amssymb, amscd}

\newcommand{\eqe}{\approx_{\epsilon}}

\newcommand\Z{{\mathbb Z}}
\newcommand\N{{\mathbb N}}

\newcommand\Q{{\mathbb Q}}

\newcommand\bp{\begin{proof}}
\newcommand\ep{\end{proof}}

\newcommand\bprop{\begin{prop}}
\newcommand\eprop{\end{prop}}

\newcommand\bt{\begin{thm}}
\newcommand\et{\end{thm}}

\newcommand\bc{\begin{cor}}
\newcommand\ec{\end{cor}}

\newcommand\ba{\begin{aligned}}
\newcommand\ea{\end{aligned}}

\newcommand\bl{\begin{lem}}
\newcommand\el{\end{lem}}

\newcommand\bi{\begin{itemize}}
\newcommand\ei{\end{itemize}}

\newcommand\br{\begin{rem}}
\newcommand\er{\end{rem}}

\newcommand\bd{\begin{defn}}
\newcommand\ed{\end{defn}}

\newcommand\Qq{{\mathcal Q}}
\newcommand\Zz{{\mathcal Z}}
\newcommand\Cc{{\mathcal C}}

\newcommand\qag{\Qq\rtimes_{\alpha}G}

\DeclareMathOperator{\Aut}{Aut}

\DeclareMathOperator{\Ad}{Ad}

\DeclareMathOperator{\id}{id}

\newtheorem{thm}{Theorem}[section]

\newtheorem*{thm*}{Theorem}

\newtheorem{prop}[thm]{Proposition}
\newtheorem{lem}[thm]{Lemma}
\newtheorem{cor}[thm]{Corollary}

\newtheorem*{q*}{Question}

\theoremstyle{remark}
\newtheorem{rem}[thm]{Remark}
\newtheorem*{rem*}{Remark}

\theoremstyle{definition}
\newtheorem{defn}[thm]{Definition}

\newtheorem{eg}[thm]{Example}
\address{ Westf\"{a}lische Wilhelms-Universit\"{a}t M\"{u}nster}
\email{sunm@uni-muenster.de}
\author{Michael Sun}
\title{Model actions for almost reduced groups on UHF algebras}
\begin{document}
\begin{abstract}
For any countable discrete group $G$ with a reduced abelian subgroup of finite index, we construct an action $\alpha$ of $G$ on the universal UHF algebra $\Qq$ using an infinite tensor product of permutation representations of $G$ and show that these actions possess some sort of Rokhlin property. The crossed product $\qag$ is then deduced to be tracially AF with a unique tracial state. We also compute the Elliott invariants in the case that $G$ is abelian.
% which we compute for abelian $G$.%in the classifiable class of unital separable simple nuclear $C^*$-algebras with tracial rank zero which satisfies the Universal Coefficient Theorem and has a unique tracial state.
%When $G$ is abelian we show $K_0(\qag)$ and $K_1(\qag)$ are equal as abelian groups and are vector spaces over $\Q$. The order on $K_0(\qag)$ is described when $G$ is finitely generated.
%abelian resin groups. existence crossed products, K theory
%classification machine: class of group actions exist. Crossed products are all classifiable%
%Characterise Residually finite groups acting via product type actions on $\Qq$. These include subgroups of direct sum of almost divisible groups. Are these all the groups?
%Compute crossed products: show they are classifiable, compute their K theory as vector spaces. what determines the dimension when the rank is infinite/

\end{abstract}
\maketitle
\section*{Introduction}
One of the first examples of the Rokhlin property that one might encounter is the following action of $\Z/2\Z$ on the UHF algebra $M_{2^{\infty}}$: First write
$$M_{2^{\infty}}=M_2\otimes M_2\otimes M_2\otimes\dots$$
and consider the order $2$ automorphism determined by
$$\alpha=\Ad\begin{pmatrix}0&1\\1&0\end{pmatrix}\otimes\Ad\begin{pmatrix}0&1\\1&0\end{pmatrix}\otimes\Ad\begin{pmatrix}0&1\\1&0\end{pmatrix}\otimes\dots$$
Then the $\Z/2\Z$-action determined by $\alpha$ has the \emph{Rokhlin property}. We might also say that $\alpha$ has the \emph{order 2} Rokhlin property as an automorphism. This action is particularly nice because it preserves a tensor product decomposition of $M_{2^{\infty}}$ and is inner on each factor. We use this example to motivate an investigation into the following class of actions. %This may be considered the class of model actions for our Rokhlin property
Let $\Qq$ be the universal UHF algebra with some tensor product decomposition
$$\Qq=M_{n_1}\otimes M_{n_2}\otimes M_{n_3}\otimes\dots.$$
Then there is an action of the product of symmetric groups
$$\prod_{l=1}^\infty S_{n_l}=S_{n_1}\times S_{n_2}\times S_{n_3}\times\dots$$
where each $S_{n_l}$ acts by conjugation by unitiaries from its natural permutation representation. By restricting this action we can define an action of any countable \emph{residually finite group}. Even simpler still is the action of the product of finite cyclic groups 
$$\prod_{l=1}^{\infty}\Z/n_l\Z =\Z/n_1\Z\times\Z/n_2\Z\times\Z/n_3\Z\times\dots$$ obtained from the previous action via a restriction of the previous representations of $S_{n_l}$ to the cyclic subgroup generated by $n_l$-cycle  $(123\dots n_l)$. Restricting to countable subgroups gives us actions for abelian groups with only trivial divisible subgroups. Such groups are often called \emph{reduced} groups. Much can be said about these actions of reduced abelian groups and their crossed products:

\begin{thm*} Suppose $G$ is a countable discrete abelian group of rank $r$ acting on $\Qq$ by an action $\alpha$ implemented by
$$G\subset \prod_{l=1}^\infty \Z/n_l\Z\to \Aut\Qq.$$
Then we have
\bi
\item $\alpha$ has the ``Rokhlin property'' if and only if $G\cap \bigoplus_{l=1}^\infty \Z/n_l\Z=0$,\\
\ei
For such $\alpha$ satisfying the ``Rokhlin property'' we have
\bi
\item[1.] $\Qq\rtimes_{\alpha}G$ is a simple $A\mathbb{T}$ algebra with a unique tracial state,\\
\item[2.] $K_i(\qag)=\Q^{2^{r-1}}$ for $i\in\{0,1\}$ if $r<\infty$.\\
\item[2.] $K_i(\qag)=\Q^{\oplus\N}$ for $i\in\{0,1\}$ if $r=\infty$.\\
\item[3.] The order on $K_0$ is determined uniquely up to isomorphism.\\
\item[4.] Every reduced abelian group has an action of this form.\\
\ei
\end{thm*}
The proof appears in Sections 2, 3 and 5, while the next result is given in Sections 4 and 5.
We say a group is \emph{almost reduced} if it has a reduced abelian subgroup of finite index. 
\begin{thm*} Suppose $G$ is a countable discrete almost reduced group. Then there exists an action $\alpha$ implemented by
$$G\subset \prod_{l=1}^\infty S_{n_l}\to \Aut\Qq$$
such that
\bi
\item $\alpha$ has the ``Rokhlin property'' and\\
\item $\qag$ is a simple nuclear TAF algebra.
\ei
\end{thm*}
In this case the actions have a nice construction as a product of induced actions of the previous abelian group actions.

As an application, these actions may be viewed as a class of model actions for the ``Rokhlin property'' with respect to certain classification results. For example, these actions contain those used to model $\Z$ actions (\cite{HO}) and finite group actions (\cite{HJ1}, \cite{HJ2}, \cite{FM1}, \cite{FM2}) with the Rokhlin property. These can also be used to model $\Z^n$ actions with the tracial Rokhlin property (classified in \cite{kishZ}, \cite{MatuiZ2}, \cite{NakaZ2}, \cite{MatuiZN}), which in particular makes it clear that they have the Rokhlin property. Modeling actions beyond abelian groups or finite groups, we have the so called Klein bottle group $\Z\rtimes\Z$ classified in \cite{MS2} and we give an explicit representative for their result here. There are many more groups here whose actions are yet to be classified.

{\bf Acknowledgements:} I would like to thank N. C. Phillips for helpful conversations. This was funded by the Research Center for Operator Algebras at East China Normal University and partially funded by SFB878 and the German Israel Foundation.

\section{Preliminaries}
\subsection{Abelian groups}

An abelian group $G$ is said to be \emph{divisible} if for every $g\in G$ and $n\in\N$ there exists $g'\in G$ such that $g=ng'$.  

\begin{lem}\label{nopdiv}\label{nodiv} If $G$ is divisible, then $G$ does not have a proper subgroup of finite index.
\end{lem}
\begin{proof} Suppose $H$ is a subgroup of $G$ with index $n$ for some $n\in\N$. Let $g\in G$. By divisibility there exists $g'\in G$ such that $g=ng'$. This means that the image of $g$ is zero in the quotient group $G/H$ and hence $g\in H$. Hence $H=G$ and $n=0$. 
\end{proof}
If a group has no divisible subgroup then it called \emph{reduced}.
Every abelian group can be decomposed uniquely as a direct sum of a divisible group and a reduced group. The Pr\"{u}fer $p$-groups $\bigcup \Z/p^r\Z$ and the rational numbers $\Q$ are the only indecomposable divisible groups.

An element $g\in G$ is said to be a \emph{torsion} element if $g$ has finite order. If there is a uniform bound on the orders of a group then we say the group is \emph{bounded}. If $G$ has no torsion elements, then $G$ is said to be \emph{torsion-free}.
%The groups$\Z/p^r\Z$ and the Pr\"{u}fer $p$-groups are torsion groups while the other groups in Table $1$ are torsion-free.

\begin{thm}[Baer]\label{baer} Suppose $G$ is a countable abelian bounded group. Then $G$ is a direct sum of finite cyclic groups.
\end{thm}
\begin{proof} See for example Kaplansky \cite{Kap}.
\end{proof}

\subsection{Reduced abelian groups}
\bd Let $\Cc$ be the class of all countable discrete abelian residually finite groups. Equivalently, $\Cc$ consists of all the countable groups for which there is an embedding into the group $G\hookrightarrow \prod_{n=1}^{\infty}\Z/n\Z$.
\ed
\bt\label{red}  $G\in \Cc$ if and only if $G$ is a countable discrete reduced abelian group.
\et
\bp Lemma \ref{nodiv} tells us there are only reduced groups in $\Cc$. We show the `if' direction.
Let $n\in\Z$. The group $G/nG$ can be seen by Theorem \ref{baer} or otherwise to be a direct sum of finite cyclic groups and is therefore residually finite. So to prove residual finiteness of $G$ it suffices to consider those $g\in G$ that have trivial image in every quotient $G/nG$. That is, $g\in\bigcap_{n\in\N}nG$, which means $g$ is the identity since $G$ is reduced.
\ep

\bc Let $\mathbb{P}$ be the set of all primes. The class $\Cc$ contains all of the subgroups of the infinite direct sum
$$\bigoplus_{r=1}^{\infty}\bigoplus_{p\in \mathbb{P}}(\Z_{(p)}\oplus\Z/p^r\Z).$$\qed
\ec
\bp
It clear that the subgroup of any group in $\Cc$ is again in $\Cc$. The class $\Cc$ is also evidently closed under taking countable direct sums by Theorem \ref{red}. So it suffices to show that $\Z_{(p)}\in\Cc$, noting $\Z/p^r\Z\in\Cc$ is trivial.

and $\Cc$ contains the countable subgroups of the $p$-adic integers for any prime $p$ by construction of the inverse limit. Therefore, $\Cc$ contains $\Z_{(p)}$ for every prime $p$ since these are subgroups of the $p$-adic integers. 
\ep
\subsection{Residually finite groups}\label{recall} A group $G$ is said to be \emph{residually finite} if for every $g\in G$ there is a a finite group $F$ and a group homomorphism $\varphi: G\to F$ such that $\varphi(g)\neq0$. Equivalently, if $G$ is countable, there is a sequence $(n_l)_{l\in\N}$ and an embedding 
$$G\hookrightarrow \prod_{l=1}^{\infty}S_{n_l},$$
where $S_{n_l}$ is the symmetric group on $n_l$.

Suppose $G$ is a group and $H$ is a subgroup. Given an $H$-set $X$,  let $\sim$ be equivalence relation on $G\times X$ generated by $(gh,x)\sim(g,hx)$. Then the induced $G$-set is defined to be
$$G\times_{H}Y= (G\times X)/\sim.$$
If we let $H$ has index $k$ and we let $g_1,\dots, g_k$ be a set of coset representatives of $H$ in $G$, then each element in $G\times_H X$ can be written uniquely as $(g_i,x)$ for $1\leq i\leq k$ and $x\in X$.

%\begin{lem}\label{gset}Let $G$ be a group and let $H$ be a subgroup of finite index $k$. Let $\{g_1,\dots,g_k\}$ be a set of coset representatives of $H$ in $G$. Let $Y$ be a $H$-set and write formally $gY=\{gy\,|\,y\in Y\}$. There is an isomorphism of $G$-sets
%$$ G\times_{H}Y\cong \bigcup_{i=1}^{k}g_iY,$$ 
%where the action on the right hand side is the natural one.
%\end{lem}
%\begin{proof} Define $f:\bigcup_{i=1}^{k}g_iY\to G\times_{H}Y$ by $g_iy\mapsto (g_i,y)$
%and define $f^-:G\times_{H}Y\to \bigcup_{i=1}^{k}g_iY$ as follows. For $g\in G$ choose $i$ and $h\in H$ such that $g=g_ih$. Then set $f^-(g,y)=g_i(hy)$. Checking that $f^-$ is well-defined and inverse to the $f$ is routine.
%\end{proof}

\subsection{Rokhlin property}
We say $\alpha\in\Aut\Qq$ is \emph{uniformly outer} if for every $a\in \Qq$, every non-zero projection $p\in \Qq$ and every $\epsilon>0$, there exists $k>0$ mutually orthogonal projections $p_1,\dots, p_k$ such that 
\bi
\item $p=p_1+\dots+p_k$ and \\
\item $p_ia\alpha(p_i)\eqe0$ for $1\leq i\leq k$.
\ei
Let $A$ be a UHF algebra, $\alpha\in\Aut A$ and let $1\leq k< \infty$. We say $\alpha$ has the order $k$ Rokhlin property, if for every $\epsilon>0$ and every finite subset $\{a_1,\dots, a_n\}$ in $A$, there exist mutually orthogonal projections $p_1,p_2,\dots ,p_{k}$ such that 
\bi
\item $[p_i,a_j]\eqe0$ for $1\leq i\leq k$ and $1\leq j\leq n$,\\
\item $\alpha(p_i)\eqe p_{i+1}$ for $1\leq i\leq k$ with $p_{k+1}=p_1$,\\
\item $p_1+\cdots+p_k=1$.
\ei
We take the order $\infty$ Rokhlin property to mean that $\alpha^j$ is uniformly outer for all $j>1$. By Kishimoto \cite{kishZ}, any two automorphisms whose powers are all uniformly outer are outer conjugate. In particular $g$ acts with the Rokhlin property since there exist uniformly outer automorphisms that act with the Rokhlin property.

\bd\label{Rok} Let $G$ be a countable discrete abelian group and $A$ be a UHF algebra. Let $g\in G$ have order $k(g)$. Then an action $\alpha$ of $G$ on $A$ has the \emph{Rokhlin property} if $\alpha_g$ has the order $k(g)$ Rokhlin property for all $g\in G$. We say an action is \emph{uniformly outer} if $\alpha_g$ is uniformly outer for all $g\in G$ except the identity.
\ed

\section{Model actions for reduced abelian groups on $\Qq$}

Let $\rho:\prod_{l=1}^{\infty}\Z/n_l\Z\to\prod_{l=1}^{\infty}U(M_{n_l})$ be the product of the natural representations.

\bprop\label{abuo} Suppose $G$ is an abelian group and $\alpha$ is an action of $G$ on $\Qq$ of the form
$$\alpha:G\stackrel{i}\hookrightarrow\prod_{l=1}^{\infty}\Z/n_l\Z\stackrel{\rho}\hookrightarrow\prod_{l=1}^{\infty}U(M_{n_l})\stackrel{\Ad}{\rightarrow}\Aut\Qq.$$
Then $\alpha$ is uniformly outer if and only if $i(G)\cap\bigoplus_{l=1}^{\infty}\Z/n_l\Z=0$.
\eprop
\bp Since the direct sum acts by inner automorphisms, necessity is clear. We now show sufficiency. Let $g\in G$ such that $i(g)$ is not in the direct sum. Let $a\in A$, let $p\in A$ be a non-zero projection and let $\epsilon>0$. Without loss of generality we can assume $\|a\|=1$ and $\epsilon$ is small. By the direct limit construction, there exists $L\in\N$ and $b,q\in M_{n_1}\otimes\dots\otimes M_{n_L}$ such that $a\approx_{\epsilon}b$ and $p\approx_{\epsilon}q$. By functional calculus we can take $q$ to be a projection.  

By assumption we can choose $l>L$ such that $i(g)_{n_l}\in U(M_{n_l})$ has order $k\neq1$. Let $e_{i,i}$ be the $i$-th diagonal matrix unit and define projections $q_1',\dots, q_k'\in M_{n_l}$ that sum to $1$ by
$$q_j'=\sum_{\lfloor ik/n_l\rfloor=j}e_{i,i}.$$
Cutting down we define $p_j'=pq_j'p$ for $1\leq j\leq k$, which are approximately projections within $2\epsilon$ and sum to $p$. By functional calculus we get projections $p_j\approx_{4\epsilon}p_j'$ in $p\Qq p$. Therefore
$$p_1+\cdots+p_k\approx_{4k\epsilon} p_1'+\cdots +p_k'=p,$$
which is the unit in $p\Qq p$. Hence $p_1+\cdots+p_k=p$, the only invertible projection in $p\Qq p$. It is now routine to check $p_ja\alpha(p_j)\approx_{13\epsilon}0$. %We have (noting $\alpha$ is isometric)
%$$\begin{aligned}p_ia\alpha(p_i)&\approx p_i'b\alpha(p_i')\\												t				&\approx q q_i'b\alpha(q_i'q)  \\
%							&= qbq_i'\alpha(q_i')\alpha(q)\\
%						&=0.\end{aligned}$$
%Hence $g$ acts via a uniformly outer automorphism.
\ep

\bt Suppose $G$ is an abelian group and $\alpha$ is an action of $G$ on $\Qq$ of the form
$$\alpha:G\stackrel{i}\hookrightarrow\prod_{l=1}^{\infty}\Z/n_l\Z\stackrel{\rho}\hookrightarrow\prod_{n=1}^{\infty}U(M_{n_l})\stackrel{\Ad}{\rightarrow}\Aut\Qq.$$
Then $\alpha$ has the  Rokhlin property if and only if $i(G)\cap\bigoplus_{l=1}^{\infty}\Z/n_l\Z=0$.
\et
\bp We prove sufficiency. Let $g\in G$. By Proposition \ref{abuo} we only need to consider the case that $g$ has finite order $k$. Let $k=P_1^{r_1}\cdots P_s^{r_s}$ be its prime factorisation into distinct primes. The assumption that no power of $i(g)$ is in the direct sum guarantees for $1\leq i \leq s$ that $P_i^{r_i}$ divides the order of $i(g)_{n_l}$ for infinitely many $l\in\N$.

Let $\epsilon>0$ and let $a_1,\dots, a_n\in \Qq$. Since $\Qq$ is a direct limit, there exists $L\in\N$ and $b_1,\dots, b_n\in M_{n_1}\otimes\dots\otimes M_{n_L}$ such that $a_i\approx_{\epsilon} b_i$ for $1\leq i\leq n$. Also for $1\leq j\leq s$, there exists $l_j>L$ all distinct such that $i(g)_{n_{l_j}}$ has order divisible by $P_j^{r_j}$. So $(i(g)_{n_{l_j}})_{1\leq j\leq s}\in\prod_{1\leq j\leq s}\Z/n_{n_{l_j}}\Z$ has order divisible by $k$. We will now be able to find $k$ projections in $M_{n_{l_1}}\otimes\cdots\otimes M_{n_{l_s}}$ to witness the order $k$ Rokhlin property.
\ep

\bprop $G\in\Cc$ if and only if there exists $(n_l)_{l\in\N}$ and an embedding $i:G\hookrightarrow\prod_{l=1}^{\infty}\Z/n_l\Z$ such that $i(G)\cap\bigoplus_{l=1}^{\infty}\Z/n_l\Z=0$.
\eprop 
\bp The `if' direction is trivial. Conversely, assume $G\in\Cc$. Then there is an embedding $i:G\hookrightarrow\prod_{l=1}^{\infty}\Z/m_l\Z$ for some sequence $(m_l)_{l\in\N}$. Now consider the diagonal embedding $i^{\N}:G\hookrightarrow (\prod_{l=1}^{\infty}\Z/m_l\Z)^{\N}$ and take $(n_l)_{l\in\N}$ to be an appropriate resequencing.
\ep
%Take the diagonal embedding of $\prod_{l=1}^{\infty}\Z/n_l\Z$ into $(\prod_{l=1}^{\infty}\Z/n_l\Z)^{\N}$, which clearly intersects the direct sum trivially. The result follows because $(\prod_{l=1}^{\infty}\Z/n_l\Z)^{\N}\cong \prod_{l=1}^{\infty}\Z/n_l\Z$ with the direct sums of each identified.

\section{Equivariant $K$-groups}

%When $G$ is finite $K_1(C^*(G))=0$ and $K_0(C^*(G))$ is the representation ring of $G$. When $G$ is finite and abelian, the irreducible representations are all one dimensional and the representation ring of $G$ is therefore isomorphic to $\Z^{|G|}$ as an abelian group. The multiplication on the representation ring is given by tensor products of representations for which the irreducible representations are invertible (since the taking of inverses is a group homomorphism for abelian groups). 

\bprop\label{finiteabelian} Suppose $\alpha$ is an action of a finite group $H$ on $\Qq$ with the Rokhlin property. Then $\Qq\rtimes_{\alpha} H\cong \Qq$.
\eprop
\bp Phillips \cite{ncp} shows that the crossed product is AF, but the proof applied here will actually show the crossed product is in fact UHF and therefore must be $\Qq$.
\ep 
%By Theorem \ref{classify} it suffices to compute the $K$-groups. Clearly $K_1$ will give zero. Let $\mathbf{1}\in M_{|G|}(\Z)$ with every entry equal to $1$, then taking $K_0$ of the connecting maps gives
%$$\Z^{|G|}\stackrel{m_n\mathbf{1}}{\longrightarrow}\Z^{|G|},$$
%where $m_n$ is a sequence of integers such that the limit of $\Z\stackrel{m_n}{\longrightarrow}\Z$ is $\Qq$. Hence the direct limit of this system is $\Q\otimes\Z[|G|^{-1}]$, which is order isomorphic to $\Qq$.
%\ep
\bt Let $G$ be a countable discrete abelian group of rank $r$ and let $\alpha$ be an action of $G$ on $\Qq$ with the  Rokhlin property of the form
$$\alpha: G\hookrightarrow\prod_{l=1}^{\infty}U(M_{n_l})\stackrel{\Ad}{\rightarrow}\Aut\Qq.$$
Then for $i\in\{0,1\}$
\bi
\item $K_i(\qag)\cong\Q^{2^{r-1}}$ if $r$ is finite.\\
\item $K_i(\qag)\cong\Q^{\oplus\N}$ if $r$ is infinite.\\
\item In either case, the inclusion of any rank $r$ free abelian subgroup $F$ induces a quasi-isomorphism $\Qq\rtimes F\hookrightarrow \Qq\rtimes G$.
\ei
%$\Qq\rtimes \Z^r\hookrightarrow\Qq\rtimes_{\alpha}G$ induces an isomorphism on $K$-theory for any %injection $\Z^r\hookrightarrow G$. In particular, $K_i(\qag)\cong\Q^{2^{r-1}}$ for $i\in\{0,1\}$ with order %depending only $r\in\N\cup\{\infty\}$. 
\et
\bp First assume $G$ is finitely generated with $G=\Z^r\oplus H$ for some finite abelian group $H$. Then $C^*(G)\cong C^*(\Z^r)\otimes C^*(H)$ and we have $\qag$ as the limit of the system
$$M_{(n_l)_{l\leq m}} \otimes C^*(\Z^r)\otimes C^*(H)\to M_{(n_l)_{l\leq m}}\otimes M_{n_{m+1}}\otimes C^*(\Z^r)\otimes C^*(H),$$
given on generators by
$$x\otimes z\otimes h\mapsto x\otimes z_{m+1}h_{m+1}\otimes z\otimes h.$$
Since $G$ is abelian we can simultaneously diagonalize the image of $C^*(G)$ in $M_{n_{m+1}}$ and
define a homotopy $t\mapsto z_{m+1}^th_{m+1}$, which extends to a homotopy of $*$-homomorphisms that connect our original map to:
%$$ M_{(n_l)_{l\leq m}} \otimes C^*(\Z^r)\otimes C^*(H)\longrightarrow M_{(n_l)_{l\leq m}}\otimes M_{n_{m+1}}C^*(\Z^r)\otimes C^*(H),$$
$$x\otimes z\otimes h\mapsto x\otimes h_{m+1}\otimes z\otimes h,$$
from which we see the limit is $C^*(\Z^r)\otimes (\Qq\rtimes_{\alpha|_H} H)$. For countable $G$, there is an increasing sequence of finitely generated subgroups $(G_n)_{n\in\N}$ of ranks $(r_n)_{n\in\N}$ such that $G=\varinjlim (G_n,\varphi_n: G_n\hookrightarrow G_{n+1})$. 
Fix decompositions $G_n=\Z^{r_{n}}\oplus H_n$ for a finite  group $H_n$. Let $k\in\N$ such that $k\varphi_n(G_n)\subset \Z^{r_{n+1}}$ and consider the following commutative diagram:
$$\begin{CD}
 	\Z^{r_n}	\oplus H_n			@>\varphi_n>>\Z^{r_{n+1}}\oplus H_{n+1}\\
@A k \oplus0AA											@A \id\oplus0 AA\\
\Z^{r_n} @>k\varphi>>\Z^{r_{n+1}}
\end{CD}$$
Take the crossed products of the actions given by restriction of the action of $G$ and consider the homotopy defined above for finitely generated groups. These are preserved by the lower three maps and we have the following commutative diagram of $C^*$-algebras:
$$\begin{CD}
 	C^*(\Z^{r_{n}})\otimes(\Qq\rtimes H_n)			@>>>C^*(\Z^{r_{n+1}})\otimes(\Qq\rtimes H_{n+1})\\
@AC^*(k)\otimes A\id u_1A											@A \id\otimes A\id u_1A\\
C^*(\Z^{r_n})\otimes\Qq@>C^*(k\varphi)\otimes\id>>C^*(\Z^{r_{n+1}})\otimes\Qq
\end{CD}$$
It suffices now to show the top map is injective on $K$-groups. We first observe the vertical maps induce isomorphisms on $K$-groups, while the bottom horizontal map is injective on $K$-groups and the result follows.%The last statement of the theorem follows from $K_i(C^*(\Z^r))=\Z^{2^{r-1}}$ and $K_i(\Qq)=\delta_{i,0}\Q$.\ep
\ep

\section{Almost reduced abelian groups on $\Qq$}
%$$\alpha: G\hookrightarrow\prod_{n=1}^{\infty}S_{n_l}\stackrel{\rho}{\rightarrow}\prod_{n=1}^{\infty}U(M_{n_l})\stackrel{\Ad}{\rightarrow}\Aut\Qq$$
%for each $h\in H$ there are infinite many $n$ for which $h$ is able to act on the diagonal projections of $M_{n_l}$ with no fixed points.

Let $\rho:\prod_{l=1}^{\infty}\Z/n_l\Z\to\prod_{l=1}^{\infty}U(M_{n_l})$ be the product of the natural representations. Recall the fibered product construction in \ref{recall}. We will use this to construct actions of groups that already have an action for one of its finite index subgroups. %that lets us extend our reduced abelian group actions to almost reduced group actions. Suppose $G$ is an almost reduced group, that is, it has a reduced abelian subgroup $H$ of finite index $k$. By Theorem \ref{}, there is an embedding 
%$$\iota:H\hookrightarrow\prod_{l=1}^{\infty}\Z/n_l\Z$$
%such that the image of $H$ does not intersect the direct sum. Each coordinate projection $\iota_{l}:H\to\Z/n_l\Z$ corresponds to the action of $H$ on a finite set $X_l$. Take the fibered product to get a finite $G$-set 
%$$Y_{l}=G\times_{H}X_{l},$$
%which in turn gives a  map $\kappa_l: G\to S_{n_lk}$ whose product gives an embedding
%$$\kappa: G\hookrightarrow\prod_{l=1}^{\infty}S_{n_lk}.$$

\bt Let $G$ be a countable discrete group with a subgroup $H$ of finite index $k$. Suppose there is an action of $H$ on $\Qq$ given by
$$\alpha_H: H\hookrightarrow\prod_{l=1}^{\infty}S_{m_l}\stackrel{\rho}{\rightarrow}\prod_{l=1}^{\infty}U(M_{n_l})\stackrel{\Ad}{\rightarrow}\Aut M_{(m_l)_{l\in\N}}$$
such that for each $h\in H$ not $1$, $\tau(\rho_l(h))=0$ for infinitely many $l\in\N$.
Then there is an action of $G$ on $\Qq$ given by
$$\alpha_H^G: G\stackrel{}{\hookrightarrow}\prod_{l=1}^{\infty}S_{n_l}\stackrel{\rho}{\rightarrow}\prod_{l=1}^{\infty}U(M_{n_l})\stackrel{\Ad}{\rightarrow}\Aut M_{(n_l)_{l\in\N}}$$
such that for each $g\in G$ not $1$, $\tau(\rho_l(g))=0$ for infinitely many $l\in\N$. Moreover, the sequence $(m_l)_{l\in\N}$ can be chosen to be of infinite type divisible by $k$ and the type of $(n_l)_{l\in\N}$.
\et
\bp Let $g\in G$ not the identity. Let $g_1,\dots, g_k$ be coset representatives for $H$. There are at most $k$ elements in $H$ for which $g_i^{-1}gg_i\in H$ for some $i$. Call this set $H_g$. Now regroup the factors for $\alpha_H$ we can arrange it so that $\tau(\rho(\iota_l(h)))=0$ for all $h\in H_g$. Since each $\rho_l$ is a permutation representation we get a sequence of finite $H$-sets $(X_l)_{l\in\N}$ such that $h$ acts with no fixed points for all $h\in H_g$. Now define a sequence $(Y_l)_{l\in\N}$ of finite $G$-sets using the fibered product construction $Y_l=G\times_HX_l$, which will not have any fixed points for $g$ by definition of $X_l$ and $H_g$. Hence we have an action given by $\alpha_H^G[g]$ of the desired form such that $\tau(\rho(\kappa_l(g)))=0$ for all $l\in\N$. Now taking a product over $g\in G$ and resequencing the factors appropriately, we get the desired result.
\ep
The vanishing trace property stated above is stronger than the property in Definition \ref{Rok} and is equivalent to it for reduced abelian groups. Taking this as the ``Rokhlin property'' we have
\bc Suppose $G$ is a countable discrete group with a reduced abelian subgroup $H$ of finite index. Then there is an action of $G$ with the ``Rokhlin property'' of the form
$$\alpha: G\stackrel{}{\hookrightarrow}\prod_{l=1}^{\infty}S_{n_l}\stackrel{\rho}{\rightarrow}\prod_{l=1}^{\infty}U(M_{n_l})\stackrel{\Ad}{\rightarrow}\Aut M_{(n_l)_{l\in\N}}.$$
\ec

\section{Classification of the crossed products $\qag$}
We show all of the crossed products considered belong to a particularly nice class of classifiable $C^*$-algebras.
%We determine the crossed products obtained from a product type action of an abelian group acting on $\Qq$ with the pointwise Rokhlin property up to isomorphism. e have a group homomorphisms $G\to U(M_{n_l})$ for each $n\in\N$. 
%Therefore we get $\qag$ as a direct limit over $n\in\N$ of the algebras $M_{n!}\otimes C^*(G)$. The map $G\to U(M_{n_l})$ gives rise to a $*$-homomorphism $\Lambda_{n_l}^G: C^*(G)\to M_{n_l}$. Let $\Delta_G$ denote the $*$-homomorphism $C^*(G)\to C^*(G)\otimes C^*(G)$ arising from the diagonal group homomorphism $G\to G\times G$. For $n\in\N$, the connecting maps can then be expressed as
%$$M_{n!}\otimes C^*(G)\stackrel{(\id\otimes\Lambda_{n+1}^G\otimes\id)(\id\otimes\Delta_G)}{\longrightarrow} M_{n!}\otimes M_{n+1}\otimes C^*(G).$$
%First we see that the limit is in a classifiable class and then we compute the $K$-groups.
\bt Suppose $G$ is a countable discrete almost abelian group and $\alpha$ is an action of $G$ on $\Qq$ with the Rokhlin property and given by
$$\alpha: G\hookrightarrow\prod_{n=1}^{\infty}U(M_{n_l})\stackrel{\Ad}{\rightarrow}\Aut\Qq.$$
Then $\qag$ is a unital separable simple nuclear tracially approximately finite dimensional locally type I $C^*$-algebra satisfying the Universal Coefficient Theorem (UCT) and has a unique tracial state.
\et
\bp We verify the conditions of Lin's \cite[Theorem 5.16]{LintypeI} to conclude that $\qag$ has tracial rank zero.  Having unique trace follows from the proof of \cite[Lemma 4.3]{kishUO}. Simplicity follows from Kishimoto \cite[Theorem 3.1]{kishO}. Nuclearity follows from Rosenberg \cite[Theorem 1]{JR}. Locally type I follows from $C^*(G)$ being type I for abelian groups and the direct limit decomposition of $\qag$. We now have real rank zero and stable rank one from R\o rdam \cite[Theorem ]{Rordam} in the presence of $\Zz$-stability, which is guaranteed from Matui-Sato \cite[Corollary 4.11]{MS2}. It also has weakly unperforated $K_0$. UCT is \cite{WELP}
\ep
%So it suffices to calculate the $K$-groups to determine the isomorphism class
%From here we only care about the $K$-groups, we can assume that the connecting maps are given by
%$$C^*(G)\stackrel{(\id\otimes\Lambda_{n+1}^G)(\Delta_G)}{\longrightarrow} M_{n!}\otimes M_{n+1}\otimes C^*(G).$$
%We consider free groups and finite groups separately and then combine them.

%\bl If $G$ is a free abelian group, then $\Lambda_{n_l}^G$ is homotopic to the trivial representation for all $n\in\N$.
%\el
%\bp Since $G$ is abelian we can assume without loss of generality that the map $G\to U(M_{n_l})$ has image contained in the diagonal matrices. For each independent generator $g\in G$, there exists $\theta_1,\dots, \theta_{n_l}\in\R$ such that $g=\diag(e^{i\theta_1},\dots,e^{i\theta_{n_l}})$ and we can define the homotopy by $g_t=\diag(e^{it\theta_1},\dots,e^{it\theta_{n_l}})$.
%\ep

%\bl Let $k$ be a strictly positive integer. Consider the directed system of abelian groups 
%$$\dots\rightarrow\Z^k\stackrel{\mathbb{1}}{\rightarrow}\Z^k\rightarrow\dots$$
%where every group is $\Z^k$ and every map $\Z^k\to\Z^k$ is given by the $k\times k$ matrix with every entry equal to $1$. Then the direct limit is isomorphic to $\Z[k^{-1}]$.
%\el
%\bp Since the image of each map is contained in the span of $(1,1,\dots,1)$ in $\Z^k$, the directed system has the same limit as one where each group is $\Z$ and the maps are multiplication by $k$. 
%\ep

\section{Examples}
\begin{eg}\label{egg} We give an action of any countable free abelian group on $\Qq$ via an action of $\Z^{\oplus\N}$ on $\Qq$. It suffices to define an action $\alpha$ of $\Z$ on $\Qq$ since then we can take tensor products to get an action of $\Z^{\oplus\N}$ on $\Qq^{\otimes\N}$. Since $\Z$ is torsion free any embedding into a product of finite cyclic groups will give a desired action by Theorem \ref{rokab}. Classification and $K$-group formulas for free abelian group actions have been given by Herman, Ocneanu, Kishimoto, Nakamura, Katsura and Matsui (\cite{HO}, \cite{kishZ}, \cite{MatuiZ2}, \cite{NakaZ2} and \cite{MatuiZN}).
\end{eg}
\begin{eg}
The Klein bottle group has presentation 
$$\Z\rtimes\Z=\langle a,b \,|\, bab^{-1}=a^{-1} \rangle.$$
% that every strongly outer action of $\Z\rtimes\Z$ on a UHF algebra $A$ that absorbs $M_{2^{\infty}}$ are equivalent and the same is true on a UHF algebra where $M_2$ does not appear in any decomposition.  They also showed the crossed product $\Qq\rtimes (\Z\rtimes\Z)$ is unital simple separable nuclear has tracial rank zero and satisfies the UCT. We will present here a model action of $\Z\rtimes\Z$ on $M_{2^{\infty}}$ with the pointwise Rokhlin property using the principles of Theorem \ref{RokQ}.
Notice that $\Z\rtimes\Z$ is an almost abelian group with subgroup $N$ generated by $a$ and $b^2$ isomorphic to $\Z\oplus\Z$. This subgroup clearly has index $2$ and is therefore normal. 

Get an action of $N=\Z\oplus\Z$ using Example \ref{egg}. For example, we could use the product of the quotient maps $\Z\oplus\Z\to \Z/p^l\Z\oplus\Z/p^l\Z$. For each quotient map, we have an action on some finite set $X_{p^l}^2$ of size $p^{2l}$ and we can take the fibered product $Y_{2p^{2l}}=(\Z\rtimes\Z)\times_{N} (X_{p^l}\times X_{p^l})$ to get a set of size $2p^{2l}$ on which $\Z\rtimes\Z$ acts. The corresponding maps $\Z\rtimes\Z\to S_{2p^{2l}}$ then combine to give the desired action.%We have a $\Z$ action on $M_{2^{\infty}}$ given by Example \ref{Zrp}. Duplicating this map we get a map $\Z\oplus \Z\to U(M_{2^l})\times U(M_{2^l})$ embedding this into $U(M_{2^{l+1}}) $we get a product type action of $\Z\oplus\Z$ on $M_{2^{\infty}}$ with the pointwise Rokhlin property. Now extend to $\Z\rtimes\Z\to U(M_{2^{l+1}}\otimes M_2)$ by induction as in Lemma \ref{induction}. Since $N$ is normal we see that this will be enough to get the pointwise Rokhlin property upon taking a direct product. Let $g\notin N$ we see that there is a basis for $\C^{2^{l+1}}\otimes\C^2$ for which $g$ acts for some $n\in N$ as
%$$\begin{pmatrix}0&1\\1&0\end{pmatrix}\otimes n.$$
%Putting this together as an infinite tensor product we have by Lemma \ref{rext} that $g$ acts with the Rokhlin property. 

Classification and $K$-group formulas were given by Matui-Sato \cite[Theorem 7.9]{MS2}
\end{eg}

\end{document}